\documentclass[portuges,12pt,letter]{article}
\usepackage[centertags]{amsmath}
\usepackage{amsfonts}
\usepackage{newlfont}
\usepackage{amscd}
\usepackage{graphics}
\usepackage{epsfig}
\usepackage{indentfirst}
\usepackage{amsxtra}
\usepackage[latin1]{inputenc}
\usepackage{amssymb, amsmath}
\usepackage{amsthm}
\usepackage[mathscr]{eucal}

\newtheorem{thm}{Theorem}[section]

\newtheorem{dfn}[thm]{Definition}

\newtheorem{remark}[thm]{Remark}

\author{Fabio Silva Botelho \\ Departamento de Matemática \\ Universidade Federal de Santa Catarina, UFSC \\
Florian\'{o}polis, SC - Brazil}

\title{\bf  A variational formulation for relativistic mechanics based on Riemannian   geometry   and its application to the quantum mechanics context}

\begin{document}
\maketitle

\abstract{This article develops a variational formulation of  relativistic nature applicable to the quantum mechanics context.
The main results are obtained through basic concepts on Riemannian geometry. Standards definitions such as vector fields and connection have a
fundamental role in the main action establishment. In the last section, as a result of an approximation for the
main formulation, we obtain the relativistic Klein-Gordon equation. }

\section{Introduction} In this article we develop a variational formulation suitable for the relativistic quantum mechanics approach in a free particle
context. The results are based on fundamental concepts of Riemannian geometry and suitable extensions for the relativistic case. Definitions such as
vector fields, connection, Lie Bracket and Riemann tensor are addressed in the subsequent sections for the main energy construction.

Indeed, the action developed in this article, in some sense, generalizes and extends the one presented in the Weinberg book \cite{50}, in chapter 12 at page 358. In such a book,
the concerned action is denoted by $I=I_M+I_G$, where $I_M$, the matter action, for $N$ particles with mass $m_n$ and charge $e_n,\; \forall n \in \{1, \ldots,N\}$, is given by
\begin{eqnarray}
I_M&=& -\sum_{n=1}^N m_n \int_{-\infty}^\infty\sqrt{-g_{\mu\nu}(x(p)) \frac{dx_n^\mu(p)}{dp}\frac{dx_n^\nu(p)}{dp}} \;dp
\nonumber \\ &&-\frac{1}{4} \int_\Omega \sqrt{g}F_{\mu\nu}F^{\mu\nu}\;d^4x \nonumber \\ &&
+\sum_{n=1}^N e_n \int_{-\infty}^\infty \frac{dx_n^\mu(p)}{dp} A_\mu(x(p))\;dp,
\end{eqnarray}
where $\{x_n(p)\}$ is the position field with concerning metrics $\{g_{\mu\nu}(x(p))\}$ and  $$F_{\mu\nu}=\frac{\partial A_\nu}{\partial x_{\mu}}-\frac{\partial A_\mu}{\partial x_{\nu}}$$ represents the  electromagnetic tensor field
through a vectorial potential $\{A_\mu\}.$

Moreover, the gravitational action $I_G$ is defined by
$$I_G=-\frac{1}{16\pi G}\int_\Omega R(x)\;\sqrt{g}\;d^4x,$$
where
$$R=g^{\mu\nu} R_{\mu\nu}.$$

Here $$R_{\mu\nu}=R_{\mu \sigma \nu}^\sigma,$$
where $$R_{\mu\sigma\nu}^\eta$$ are the components of the well known Riemann curvature tensor.

According to \cite{50}, the Euler-Lagrange equations for $I$ correspond to the Einstein field equations,
$$R^{\mu\nu}-\frac{1}{2} g^{\mu\nu}R+8\pi G T^{\mu\nu}=0,$$
where the energy-momentum tensor $T^{\mu\nu}$ is expressed by
\begin{eqnarray}
T^{\lambda \kappa} &=&\sqrt{g} \sum_{n=1}^N \int_{-\infty}^\infty\frac{dx_n^\lambda(p)}{d\tau_n}\frac{dx_n^\kappa(p)}{d \tau_n}\;\delta^4(x-x_n)\;d\tau_n
\nonumber \\ && +F_\mu^\lambda(x)F^{\mu \kappa}(x)-\frac{1}{4} g^{\lambda \kappa}F_{\mu\nu}F^{\mu\nu}.
\end{eqnarray}

One of the main differences of our model from this previous one, is that we consider a possible variation in the density along the mechanical system.

Also, in our model, the motion of the system in question is specified by a four-dimensional manifold given by the function $$\mathbf{r}(\hat{\mathbf{u}}(\mathbf{x},t))=(ct,X_1(\mathbf{u}(\mathbf{x},t)),X_2(\mathbf{u}(\mathbf{x},t)),X_3(\mathbf{u}(\mathbf{x},t))),$$
 with corresponding mass density $$(\rho \circ \hat{\mathbf{u}}): \Omega \times [0,T] \rightarrow \mathbb{R}^+,$$
where $\Omega \subset \mathbb{R}^3$ and $[0,T]$ is a time interval. At this point, we define $\phi(\hat{\mathbf{u}}(\mathbf{x},t))$ as a complex function such that $$|\phi(\hat{\mathbf{u}}(\mathbf{x},t))|^2
=\frac{\rho(\hat{\mathbf{u}}(\mathbf{x},t))}{m},$$ where $m$ denotes the total system mass at  rest.
We emphasize it seems to be clear that in the previous book the parametrization of the position field, through the parameter $p$, is one-dimensional.

In this work we do not consider the presence of electromagnetic fields.

Anyway, the final expression  of the related new action here developed is given by
\begin{eqnarray}
&& J(\phi,\mathbf{r}, \hat{\mathbf{u}},E)\nonumber \\ &=&\int_0^T \int_\Omega mc\;\sqrt{-g_{ij}\frac{\partial u_i}{\partial t} \frac{\partial u_j}{\partial t}}\;|\phi(\hat{\mathbf{u}}(\mathbf{x},t))|^2\; \sqrt{-g}\;|\det (\hat{\mathbf{u}}'(\mathbf{x},t))|\;d\mathbf{x}\;dt
 \nonumber \\ &&+\frac{\gamma}{2}\int_0^T\int_\Omega g^{jk}\;\frac{\partial \phi }{\partial u_j}\;\frac{\partial \phi^* }{\partial u_k}\;\sqrt{-g}\;|\det(\hat{\mathbf{u}}'(\mathbf{x},t))|\;d\mathbf{x}dt
\nonumber \\ &&+
\frac{\gamma}{4}\int_0^T\int_\Omega g^{jk}\;\left(\phi^* \frac{\partial \phi }{\partial u_l}+\phi \frac{\partial \phi^* }{\partial u_l}\right)\; \Gamma_{jk}^l\;\sqrt{-g}\;|\det(\hat{\mathbf{u}}'(\mathbf{x},t))|\;d\mathbf{x}dt
 \nonumber \\ &&  +\frac{\gamma}{2}\int_0^T\int_\Omega |\phi|^2\;g^{jk}\;\left(\frac{\partial \Gamma_{lk}^l}{\partial u_j}
 -\frac{\partial \Gamma_{jk}^l}{\partial u_l}+\Gamma_{lk}^p \;\Gamma_{jp}^l-\Gamma_{jk}^p \;\Gamma_{lp}^l\right) \;\sqrt{-g}\;|\det(\hat{\mathbf{u}}'(\mathbf{x},t))|\;d\mathbf{x}dt \nonumber \\ &&-\int_0^T E(t)\left( \int_\Omega |\phi(\hat{\mathbf{u}}(\mathbf{x},t))|^2\;\sqrt{-g}\; |\det(\hat{\mathbf{u}}'(\mathbf{x},t))|\;
d\mathbf{x}/c-1\right)\;cdt.
 \end{eqnarray}
 Observe that the action part\begin{eqnarray}&&\frac{\gamma}{2}\int_0^T\int_\Omega |\phi|^2\;g^{jk}\;\left(\frac{\partial \Gamma_{lk}^l}{\partial u_j}
 -\frac{\partial \Gamma_{jk}^l}{\partial u_l}+\Gamma_{lk}^p \;\Gamma_{jp}^l-\Gamma_{jk}^p \;\Gamma_{lp}^l\right) \;\sqrt{-g}\;|\det(\hat{\mathbf{u}}'(\mathbf{x},t))|\;d\mathbf{x}dt \nonumber \\ &=&\frac{\gamma}{2}\int_0^T\int_\Omega |\phi|^2 \hat{R} \;\sqrt{-g}\;|\det(\hat{\mathbf{u}}'(\mathbf{x},t))|\;d\mathbf{x}dt,\end{eqnarray}
where $$\hat{R}=g^{jk}\hat{R}_{jk},$$  $$\hat{R}_{jk}=\hat{R}_{jlk}^l$$ and $$\hat{R}_{ijk}^l=\frac{\partial \Gamma_{jk}^l}{\partial u_i}-\frac{\partial \Gamma_{ik}^l}{\partial u_j}+\Gamma_{jk}^p \;\Gamma_{ip}^l-\Gamma_{ik}^p \;\Gamma_{jp}^l$$ represents the Riemann curvature tensor,
  corresponds tho the Hilbert-Einstein one, as specified in the subsequent sections.

In the last section, we show how such a formulation may result, as an approximation,  the  well known relativistic Klein-Gordon one and the respective Euler-Lagrange  equations. We believe the main results obtained may be extended to more complex mechanical systems, including in some  extent, the  quantum
mechanics approach.

 Finally, about the references, details on the Sobolev Spaces involved may be found in \cite{1,12}. For standard references in quantum mechanics, we refer to
\cite{55,1410,101} and the non-standard \cite{100}.

\section{ Some introductory topics on vector analysis and  Riemannian geometry}

In this section we   present some introductory remarks on Riemannian geometry.
The results  here developed have been presented in details in \cite{19}. For the sake of completeness, we repeat some of the proofs.

We start with the definition of surface in $\mathbb{R}^n.$

\begin{dfn}[Surface in $\mathbb{R}^n$, the respective tangent space and the dual one] Let $D \subset \mathbb{R}^m$ be an open, bounded, connected set with a regular (Lipschitzian)
boundary denoted by $\partial D$. We define a m-dimensional $C^1$ class surface $M \subset \mathbb{R}^n$, where $1 \leq m <n$, as the range of a function
$\mathbf{r}:D \subset \mathbb{R}^m \rightarrow \mathbb{R}^n,$ where
$$M=\{\mathbf{r}(\mathbf{u})\;:\; \mathbf{u}=(u_1,\ldots,u_m) \in D\}$$
and
$$\mathbf{r}(\mathbf{u})=\hat{X}_1(\mathbf{u})\mathbf{e}_1+\hat{X}_2(\mathbf{u})\mathbf{e}_ 2+\cdots +\hat{X}_n(\mathbf{u})\mathbf{e}_n,$$
where $\hat{X}_k:D \rightarrow \mathbb{R}$ is a $C^1$ class function, $\forall k \in \{1, \ldots, n\},$ and $\{\mathbf{e}_1, \ldots, \mathbf{e}_n\}$
is the canonical basis of $\mathbb{R}^n.$

Let $\mathbf{u} \in D$ and  $p=\mathbf{r}(\mathbf{u}) \in M.$ We also define the tangent space of $M$ at $p$, denoted by $T_p(M)$, as
$$T_p(M)=\left\{ \alpha_1 \frac{\partial \mathbf{r}(\mathbf{u})}{\partial u_1}+\cdots + \alpha_m \frac{\partial \mathbf{r}(\mathbf{u})}{\partial u_m}
\;:\; \alpha_1, \ldots, \alpha_m \in \mathbb{R}\right\}.$$

We assume $$\left\{\frac{\partial \mathbf{r}(\mathbf{u})}{\partial u_1}, \cdots, \frac{\partial \mathbf{r}(\mathbf{u})}{\partial u_m}\right\}$$
to be a linearly independent set $\forall \mathbf{u} \in D.$

Finally, we define the dual space to $T_p(M)$, denoted by $T_p(M)^*$, as the set of all continuous and linear functionals (in fact real functions)
defined on $T_p(M),$ that is,
$$T_p(M)^*=\left\{ f:T_p(M) \rightarrow \mathbb{R}\;:\; f(\mathbf{v})=\alpha \cdot \mathbf{v},\; \text{ for  some } \alpha \in \mathbb{R}^n,\; \forall
\mathbf{v}=v_i \frac{\partial \mathbf{r}(\mathbf{u})}{\partial u_i} \in T_p(M)\right\}.$$
\end{dfn}
\begin{thm} Let $M \subset \mathbb{R}^n$ be a m-dimensional $C^1$ class surface, where
$$M=\{ \mathbf{r}(\mathbf{u}) \in \mathbb{R}^n\;:\; \mathbf{u} \in D \subset \mathbb{R}^m\}.$$

Let $\mathbf{u} \in D$, $p=\mathbf{r}(\mathbf{u}) \in D$ and $f \in C^1(M).$

Define $df: T_p(M) \rightarrow \mathbb{R}$ by
$$df(\mathbf{v})=\lim_{ \varepsilon \rightarrow 0}
\frac{(f \circ \mathbf{r})(\{u_i\}+\varepsilon \{v_i\})-(f \circ \mathbf{r})(\{u_i\})}{\varepsilon},$$
$\forall \mathbf{v}= v_i \frac{\partial \mathbf{r}(\mathbf{u})}{\partial u_i} \in T_p(M).$

Under such hypotheses, $$df \in T_p(M)^*.$$

Reciprocally, let $F \in T_p(M)^*.$

Under such assumption, there exists $f \in C^1(M)$ such that $$F(\mathbf{v})=df(\mathbf{v}),\; \forall \mathbf{v} \in T_p(M).$$
\end{thm}
\begin{proof} Let $\mathbf{v}= v_i \frac{\partial \mathbf{r}(\mathbf{u})}{\partial u_i} \in T_p(M).$

Thus,
\begin{eqnarray}
df(\mathbf{v})&=&\lim_{ \varepsilon \rightarrow 0}
\frac{(f \circ \mathbf{r})(\{u_i\}+\varepsilon \{v_i\})-(f \circ \mathbf{r})(\{u_i\})}{\varepsilon}
\nonumber \\ &=& \frac{\partial (f \circ \mathbf{r})(\mathbf{u})}{\partial \hat{X}_j} \frac{\partial \hat{X}_j(\mathbf{u})}{\partial u_i} v_i
\nonumber \\ &=& \nabla f(\mathbf{r}(\mathbf{u})) \cdot \mathbf{v} \nonumber \\ &=& \alpha \cdot \mathbf{v}, \end{eqnarray}
where $$\alpha=\nabla f(\mathbf{r}(\mathbf{u})),$$ so that $df \in T_p(M)^*.$

Reciprocally, assume $F \in T_p(M)^*,$ that is, suppose there exists $\alpha \in \mathbb{R}^n$ such that $$F(\mathbf{v})= \alpha \cdot \mathbf{v},$$
$\forall  \mathbf{v}= v_i \frac{\partial \mathbf{r}(\mathbf{u})}{\partial u_i} \in T_p(M).$

Define $f: M \rightarrow \mathbb{R}$ by $$f(\mathbf{w})=\alpha \cdot \mathbf{w},\; \forall \mathbf{w} \in M.$$

In particular, $$f(\mathbf{r}(\mathbf{u}))= \alpha \cdot \mathbf{r}(\mathbf{u})=\alpha_j \hat{X}_j(\mathbf{u}),\; \forall \mathbf{u} \in D.$$

For $p=\mathbf{r}(\mathbf{u}) \in M$ and $\mathbf{v}= v_i \frac{\partial \mathbf{r}(\mathbf{u})}{\partial u_i} \in T_p(M)$, we have
\begin{eqnarray}df(\mathbf{v})&=&\lim_{ \varepsilon \rightarrow 0}
\frac{(f \circ \mathbf{r})(\{u_i\}+\varepsilon \{v_i\})-(f \circ \mathbf{r})(\{u_i\})}{\varepsilon}
\nonumber \\ &=&\frac{\partial (f \circ \mathbf{r})(\mathbf{u})}{\partial \hat{X}_j} \frac{\partial \hat{X}_j(\mathbf{u})}{\partial u_i} v_i
\nonumber \\ &=& \alpha_j\frac{\partial \hat{X}_j(\mathbf{u})}{\partial u_i} v_i \nonumber \\ &=& \alpha \cdot \mathbf{v}, \end{eqnarray}

Therefore, $$F(\mathbf{v})=df(\mathbf{v}),\; \forall \mathbf{v} \in T_p(M).$$

The proof is complete.
\end{proof}

At this point, we present the tangential vector field definition, to be addressed in the subsequent results and sections.
\begin{dfn}[Vector field] Let $M \subset \mathbb{R}^n$ be a m-dimensional $C^1$ class surface, where $1 \leq m < n.$ We define the set of
$C^1$ class tangential vector fields in $M$, denoted by $\mathcal{X}(M)$, as
$$\mathcal{X}(M)=\left\{X=X_i(\mathbf{u})\frac{\partial \mathbf{r}(\mathbf{u})}{\partial u_i} \in T(M)=\{T_p(M)\;:\; p=\mathbf{r}(\mathbf{u}) \in M\}\right\},$$
where $X_i:D \rightarrow \mathbb{R}$ is a $C^1$ class function, $\forall i \in \{1, \ldots,m\}.$

Let $f \in C^1(M)$ and $X \in \mathcal{X}(M)$. We define the derivative of $f$ on the direction $X$ at $\mathbf{u}$, denoted by $(X \cdot f)(p)$,
where $p=\mathbf{r}({\mathbf{u}})$, as
\begin{eqnarray}(X \cdot f)(p)&=& df(X(\mathbf{u})) \nonumber \\ &=&
\lim_{ \varepsilon \rightarrow 0} \frac{ (f\circ \mathbf{r})(\{u_i\}+\varepsilon \{X_i(\mathbf{u})\})-(f \circ \mathbf{r})(\{u_i\})}{\varepsilon}
\nonumber \\ &=& \frac{\partial (f \circ \mathbf{r})(\mathbf{u})}{\partial u_i} X_i(\mathbf{u}).
\end{eqnarray}
\end{dfn}
The next definition is also very important for this work, namely, the connection one.
\begin{dfn}[Connection]Let $M \subset \mathbb{R}^n$ be a m-dimensional $C^1$ class surface, where
$$M=\{ \mathbf{r}(\mathbf{u}) \in \mathbb{R}^n\;:\; \mathbf{u} \in D \subset \mathbb{R}^m\}$$
and $$\mathbf{r}(\mathbf{u})=\hat{X}_1(\mathbf{u})\mathbf{e}_1+\cdots +\hat{X}_n(\mathbf{u})\mathbf{e}_n.$$

We define an affine connection on $M$, as a map $\nabla : \mathcal{X}(M) \times \mathcal{X}(M) \rightarrow \mathcal{X}(M)$ such that
\begin{enumerate}
\item $$\nabla_{fX+gY} Z=f \nabla_X Z+g \nabla_Y Z,$$
\item $$\nabla_X(Y+Z)=\nabla_X Y+\nabla_X Z$$ and
\item $$\nabla_X(fY)=(X\cdot f)Y+f \nabla_XY,$$
\end{enumerate}
$\forall X,Y,Z \in \mathcal{X}(M),\; f,g \in C^{\infty}(M).$
\end{dfn}

About the connection representation, we have the following result.
\begin{thm}Let $M \subset \mathbb{R}^n$ be a m-dimensional $C^1$ class surface, where
$$M=\{ \mathbf{r}(\mathbf{u}) \in \mathbb{R}^n\;:\; \mathbf{u} \in D \subset \mathbb{R}^m\}$$
and $$\mathbf{r}(\mathbf{u})=\hat{X}_1(\mathbf{u})\mathbf{e}_1+\cdots +\hat{X}_n(\mathbf{u})\mathbf{e}_n.$$
Let $\nabla  : \mathcal{X}(M) \times \mathcal{X}(M) \rightarrow \mathcal{X}(M)$ be an affine connection on $M$.
Let $\mathbf{u} \in D$, $p=\mathbf{r}(\mathbf{u}) \in M$ and $X,Y \in \mathcal{X}(M)$ be such that
$$X= X_i(\mathbf{u}) \frac{\partial \mathbf{r}(\mathbf{u})}{\partial u_i},$$
and
$$Y= Y_i(\mathbf{u}) \frac{\partial \mathbf{r}(\mathbf{u})}{\partial u_i}.$$

Under such hypotheses, we have
 \begin{eqnarray}\nabla_XY&=&\sum_{i=1}^m\left(X \cdot Y_i+\sum_{j,k=1}^m \Gamma_{jk}^iX_j Y_k\right)\frac{\partial \mathbf{r}(\mathbf{u})}{\partial u_i} \in T_p(M),\end{eqnarray}
 where $\Gamma_{jk}^i$ are defined through the relations,

 $$\nabla_{\frac{\partial \mathbf{r}(\mathbf{u})}{\partial u_j}}\frac{\partial \mathbf{r}(\mathbf{u})}{\partial u_k}
 = \Gamma_{jk}^i(\mathbf{u}) \frac{\partial \mathbf{r}(\mathbf{u})}{\partial u_i}.$$
 \end{thm}
 \begin{proof}
 Observe that
 \begin{eqnarray}
 \nabla_XY&=& \nabla_{X_i \frac{\partial \mathbf{r}(\mathbf{u})}{\partial u_i}} \left(Y_j \frac{\partial \mathbf{r}(\mathbf{u})}{\partial u_j}\right) \nonumber \\ &=&
 X_i \nabla_{\frac{\partial \mathbf{r}(\mathbf{u})}{\partial u_i}}\left(Y_j \frac{\partial \mathbf{r}(\mathbf{u})}{\partial u_j}\right) \nonumber \\ &=& X_i\left( \frac{\partial \mathbf{r}(\mathbf{u})}{\partial u_i} \cdot Y_j\right) \frac{\partial \mathbf{r}(\mathbf{u})}{\partial u_j}+X_iY_j \nabla_{\frac{\partial \mathbf{r}(\mathbf{u})}{\partial u_i}} \frac{\partial \mathbf{r}(\mathbf{u})}{\partial u_j} \nonumber \\ &=&
 \left( X_i\frac{\partial \mathbf{r}(\mathbf{u})}{\partial u_i} \cdot Y_j\right) \frac{\partial \mathbf{r}(\mathbf{u})}{\partial u_j}+X_iY_j  \Gamma_{ij}^k \frac{\partial \mathbf{r}(\mathbf{u})}{\partial u_k} \nonumber \\ &=&
 \sum_{i=1}^m\left(X \cdot Y_i+\sum_{j,k=1}^m \Gamma_{jk}^iX_j Y_k\right)\frac{\partial \mathbf{r}(\mathbf{u})}{\partial u_i}. \end{eqnarray}

  The proof is complete.
\end{proof}
\begin{remark} If the connection in question is such that
$$\Gamma_{jk}^i=\frac{1}{2}g^{il}\left(\frac{\partial g_{kl}}{\partial u_j}+\frac{\partial g_{jl}}{\partial u_k}-\frac{\partial g_{jk}}{\partial u_l} \right)$$ such a connection is said to be the Levi-Civita one. In the next lines we assume the concerning connection is indeed the Levi-Civita one.
\end{remark}

We finish this section with the Lie Bracket definition.
\begin{dfn}[Lie bracket] Let $M \subset \mathbb{R}^n$ be a $C^1$ class  m-dimensional surface where $1 \leq m < n.$

Let $X,Y \in \tilde{\mathcal{X}}(M),$  where $\tilde{\mathcal{X}}(M)$ denotes the set of the $C^\infty(M)=\cap_{k \in \mathbb{N}}C^k(M)$ class vector fields. We define the Lie bracket of $X$ and $Y$, denoted by $[X,Y] \in \tilde{\mathcal{X}}(M)$, by
$$[X,Y]=(X\cdot Y_i-Y\cdot X_i) \frac{\partial \mathbf{r}(\mathbf{u})}{\partial u_i},$$
where
$$X=X_i \frac{\partial \mathbf{r}(\mathbf{u})}{\partial u_i}$$ and
$$Y=Y_i \frac{\partial \mathbf{r}(\mathbf{u})}{\partial u_i}.$$
\end{dfn}

\section{ A relativistic quantum mechanics action}

In this section we present a proposal for a relativistic quantum mechanics action.

Let $\Omega \subset \mathbb{R}^3$ be an open, bounded,  connected set with a $C^1$ class boundary denoted by $\partial \Omega.$
Denoting by $c$ the speed of light, in a free volume context, for a $C^1$ class function $\mathbf{r}$ and $\hat{\mathbf{u}} \in W^{1,2}(\Omega \times [0,T]; \mathbb{R}^4)$, let $(\mathbf{r} \circ \hat{\mathbf{u}}):\Omega \times [0,T]  \rightarrow \mathbb{R}^4$ be a particle position field where
$$\mathbf{r}(\hat{\mathbf{u}}(\mathbf{x},t))=(ct,X_1(\mathbf{u}(\mathbf{x},t)),X_2(\mathbf{u}(\mathbf{x},t)),X_3(\mathbf{u}(\mathbf{x},t))),$$
 with corresponding mass density $$(\rho \circ \hat{\mathbf{u}}): \Omega \times [0,T] \rightarrow \mathbb{R}^+,$$
where $[0,T]$ is a time interval.

We  denote $\hat{\mathbf{u}}: \Omega \times [0,T] \rightarrow \mathbb{R}^4$ point-wise as
$$\hat{\mathbf{u}}(\mathbf{x},t)=(u_0(\mathbf{x},t), \mathbf{u}(\mathbf{x},t)),$$ where $$u_0(\mathbf{x},t)=ct,$$ and
$$\mathbf{u}(\mathbf{x},t)=(u_1(\mathbf{x},t), u_2(\mathbf{x},t),u_3(\mathbf{x},t)),$$
$\forall (\mathbf{x},t)=((x_1,x_2,x_3),t) \in \Omega \times [0,T].$

At this point, we recall to have defined  $\phi(\hat{\mathbf{u}}(\mathbf{x},t))$ as a complex function such that $$|\phi(\hat{\mathbf{u}}(\mathbf{x},t))|^2
=\frac{\rho(\hat{\mathbf{u}}(\mathbf{x},t))}{m},$$ where $m$ denotes the total system mass at  rest. Also, we assume $\phi$ to be a $C^2$ class function and
define
$$d \tau^2=c^2 dt^2-dX_1(\mathbf{u}(\mathbf{x},t))^2-dX_2(\mathbf{u}(\mathbf{x},t))^2-dX_3(\mathbf{u}(\mathbf{x},t))^2,$$
so that the mass differential will be denoted by
\begin{eqnarray}dm&=&\frac{\rho(\hat{\mathbf{u}}(\mathbf{x},t))}{\sqrt{1-\frac{v^2}{c^2}}} \sqrt{-g}\;|\det (\hat{\mathbf{u}}'(\mathbf{x},t))|\;d\mathbf{x}
\nonumber \\ &=&\frac{m |\phi(\hat{\mathbf{u}}(\mathbf{x},t))|^2}{\sqrt{1-\frac{v^2}{c^2}}} \sqrt{-g}\;|\det (\hat{\mathbf{u}}'(\mathbf{x},t))|\;d\mathbf{x}
,\end{eqnarray}
where $d \mathbf{x}=dx_1dx_2dx_3$ and $\hat{\mathbf{u}}'(\mathbf{x},t)$ denotes the Jacobian matrix of the vectorial function $\hat{\mathbf{u}}(\mathbf{x},t).$

Also,
$$\mathbf{g}_i=\frac{\partial \mathbf{r}(\hat{\mathbf{u}})}{\partial u_i},\; \forall i \in \{0,1,2,3\},$$
 $$g_{ij}=\mathbf{g}_i \cdot \mathbf{g}_j,\; \forall i,j \in \{0,1,2,3\},$$
and $$g=\det\{g_{ij}\}.$$

Moreover, $$\{g^{ij}\}=\{g_{ij}\}^{-1}.$$
\subsection{The kinetics energy}
Observe  that \begin{eqnarray}
c^2-v^2&=&-\frac{d\mathbf{r}(\hat{\mathbf{u}})}{dt} \cdot \frac{d\mathbf{r}(\hat{\mathbf{u}})}{dt}
\nonumber \\ &=& -\left(\frac{\partial \mathbf{r}(\hat{\mathbf{u}})}{\partial u_i}\frac{\partial u_i}{\partial t}\right)\cdot
\left(\frac{\partial \mathbf{r}(\hat{\mathbf{u}})}{\partial u_j}\frac{\partial u_j}{\partial t}\right)
\nonumber \\ &=& -\frac{\partial \mathbf{r}(\hat{\mathbf{u}})}{\partial u_i} \cdot\frac{\partial \mathbf{r}(\hat{\mathbf{u}})}{\partial u_j}
 \frac{\partial u_i}{\partial t} \frac{\partial u_j}{\partial t} \nonumber \\ &=& -g_{ij}\frac{\partial u_i}{\partial t} \frac{\partial u_j}{\partial t},
 \end{eqnarray}
where the  product in question is generically given by $$ \mathbf{y} \cdot \mathbf{z} =-y_0z_0+\sum_{i=1}^3y_i z_i, \; \forall \mathbf{y}=(y_0,y_1,y_2,y_3), \; \mathbf{z}=(z_0,z_1,z_2,z_3) \in \mathbb{R}^4$$
and
 $$v=\sqrt{\left(\frac{dX_1(\mathbf{u}(\mathbf{x},t))}{dt}\right)^2+\left(\frac{dX_2(\mathbf{u}(\mathbf{x},t))}{dt}\right)^2
+\left(\frac{dX_3(\mathbf{u}(\mathbf{x},t))}{dt}\right)^2}.$$

The semi-classical kinetics energy differential is given by   \begin{eqnarray}dE_c&=& \frac{d \mathbf{r}(\hat{\mathbf{u}})}{d t}\cdot \frac{d \mathbf{r}(\hat{\mathbf{u}})}{d t} \;dm\nonumber \\ &=& -\left(\frac{d\tau}{dt}\right)^2\;dm \nonumber \\ &=&-(c^2-v^2)\;dm,\end{eqnarray} so that
 \begin{eqnarray}dE_c&=&  -m\;\frac{(c^2-v^2)}{\sqrt{1-\frac{v^2}{c^2}}}\;|\phi(\hat{\mathbf{u}})|^2\; \sqrt{-g}\;|\det (\hat{\mathbf{u}}'(\mathbf{x},t))|\;d\mathbf{x} \nonumber \\ &=&
-m c^2\sqrt{1-\frac{v^2}{c^2}}\;|\phi(\hat{\mathbf{u}})|^2\; \sqrt{-g}\;|\det (\hat{\mathbf{u}}'(\mathbf{x},t))|\;d\mathbf{x}
\nonumber \\ &=& -mc\;\sqrt{c^2-v^2}\;|\phi(\hat{\mathbf{u}})|^2\; \sqrt{-g}\;|\det (\hat{\mathbf{u}}'(\mathbf{x},t))|\;d\mathbf{x} \nonumber \\ &=&
-mc\;\sqrt{-g_{ij}\frac{\partial u_i}{\partial t} \frac{\partial u_j}{\partial t}}\;|\phi(\hat{\mathbf{u}})|^2 \;\sqrt{-g}\;|\det (\hat{\mathbf{u}}'(\mathbf{x},t))|\;d\mathbf{x}, \end{eqnarray}
and thus, the semi-classical kinetics energy $E_c$ is given by
$$E_c=\int_0^T\int_\Omega dE_c\;dt,$$
that is,
$$E_c=-\int_0^T \int_\Omega mc\;\sqrt{-g_{ij}\frac{\partial u_i}{\partial t} \frac{\partial u_j}{\partial t}}\;|\phi(\hat{\mathbf{u}})|^2\; \sqrt{-g}\;|\det (\hat{\mathbf{u}}'(\mathbf{x},t))|\;d\mathbf{x}\;dt.$$

\subsection{The energy part relating the curvature and wave function}

At this point we define an energy part, related to the  Riemann curvature tensor,  denoted by $E_q$, where
$$E_q=\frac{\gamma}{2} \int_0^T\int_\Omega g^{jk} R_{jk}
\sqrt{-g}\; |\det(\hat{\mathbf{u}}'(\mathbf{x},t))|\;d\mathbf{x}\;dt.$$
and $$R_{jk}=Re[R_{jik}^i(\phi)].$$  Also, generically $Re[z]$ denotes the real part of $z \in \mathbb{C}$ and $R_{ijk}^l(\phi)$ is such that
\begin{eqnarray}&&\nabla_{\left(\phi \frac{\partial \mathbf{r}(\hat{\mathbf{u}})}{\partial u_i}\right)}\nabla_{ \frac{\partial \mathbf{r}(\hat{\mathbf{u}})}{\partial u_j}}
\left(\phi^* \frac{\partial \mathbf{r}(\hat{\mathbf{u}})}{\partial u_k}\right)-\nabla_{\left(\phi \frac{\partial \mathbf{r}(\hat{\mathbf{u}})}{\partial u_j}\right)}\nabla_{ \frac{\partial \mathbf{r}(\hat{\mathbf{u}})}{\partial u_i}}
\left(\phi^* \frac{\partial \mathbf{r}(\hat{\mathbf{u}})}{\partial u_k}\right)
\nonumber \\ &&
-\nabla_{\left[ \phi \frac{\partial \mathbf{r}(\hat{\mathbf{u}})}{\partial u_i},\frac{\partial \mathbf{r}(\hat{\mathbf{u}})}{\partial u_j}\right]}
\left(\phi^* \frac{\partial \mathbf{r}(\hat{\mathbf{u}})}{\partial u_k}\right)
=R^l_{ijk}(\phi) \frac{\partial \mathbf{r}(\hat{\mathbf{u}})}{\partial u_l}.\end{eqnarray}

More specifically, we have
\begin{eqnarray} &&\nabla_{\left(\phi \frac{\partial \mathbf{r}(\hat{\mathbf{u}})}{\partial u_i}\right)}\;\nabla_{ \frac{\partial \mathbf{r}(\hat{\mathbf{u}})}{\partial u_j}}\;
\left(\phi^* \frac{\partial \mathbf{r}(\hat{\mathbf{u}})}{\partial u_k}\right) \nonumber \\
&=& \nabla_{\left(\phi\; \frac{\partial \mathbf{r}(\hat{\mathbf{u}})}{\partial u_i}\right)}\;\left(\frac{\partial \phi^*}{\partial u_j}
\;\frac{\partial \mathbf{r}(\hat{\mathbf{u}})}{\partial u_k}+ \phi^*\;\Gamma_{jk}^l\;\frac{\partial \mathbf{r}(\hat{\mathbf{u}})}{\partial u_l}\right)
\nonumber \\ &=& \phi\;\frac{\partial^2 \phi^*}{\partial u_i\partial u_j}\;\frac{\partial \mathbf{r}(\hat{\mathbf{u}})}{\partial u_k}
+\phi \;\frac{\partial \phi^*}{\partial u_j}\; \Gamma_{ik}^p\; \frac{\partial \mathbf{r}(\mathbf{u})}{\partial u_p}
\nonumber \\ &&+ \phi \;\frac{\partial \left(\phi^* \;\Gamma_{jk}^l\right)}{\partial u_i}\;\frac{\partial \mathbf{r}(\hat{\mathbf{u}})}{\partial u_l}
\nonumber \\ &&+|\phi|^2\;\Gamma_{ij}^l\; \Gamma_{il}^p \;\frac{\partial \mathbf{r}(\hat{\mathbf{u}})}{\partial u_p} \nonumber \\ &=&
\phi\;\frac{\partial^2 \phi^*}{\partial u_i\partial u_j}\;\delta_{kl}\;\frac{\partial \mathbf{r}(\hat{\mathbf{u}})}{\partial u_l}
+\phi \frac{\partial \phi^*}{\partial u_j} \Gamma_{ik}^l \frac{\partial \mathbf{r}(\hat{\mathbf{u}})}{\partial u_l}
\nonumber \\ &&+ \phi\; \frac{\partial (\phi^* \Gamma_{jk}^l)}{\partial u_i}\;\frac{\partial \mathbf{r}(\hat{\mathbf{u}})}{\partial u_l}
\nonumber \\ &&+|\phi|^2\;\Gamma_{jk}^p \;\Gamma_{ip}^l \;\frac{\partial \mathbf{r}(\hat{\mathbf{u}})}{\partial u_l} \end{eqnarray}
and similarly
\begin{eqnarray} &&\nabla_{\left(\phi \frac{\partial \mathbf{r}(\hat{\mathbf{u}})}{\partial u_j}\right)}\;\nabla_{ \frac{\partial \mathbf{r}(\hat{\mathbf{u}})}{\partial u_i}}\;
\left(\phi^* \frac{\partial \mathbf{r}(\hat{\mathbf{u}})}{\partial u_k}\right) \nonumber \\ &=&
\phi\;\frac{\partial^2 \phi^*}{\partial u_i\partial u_j}\;\delta_{kl}\;\frac{\partial \mathbf{r}(\hat{\mathbf{u}})}{\partial u_l}
+\phi \frac{\partial \phi^*}{\partial u_i} \Gamma_{jk}^l \frac{\partial \mathbf{r}(\hat{\mathbf{u}})}{\partial u_l}
\nonumber \\ &&+ \phi\; \frac{\partial (\phi^* \Gamma_{ik}^l)}{\partial u_j}\;\frac{\partial \mathbf{r}(\hat{\mathbf{u}})}{\partial u_l}
\nonumber \\ &&+|\phi|^2\;\Gamma_{ik}^p \;\Gamma_{jp}^l \;\frac{\partial \mathbf{r}(\hat{\mathbf{u}})}{\partial u_l}. \end{eqnarray}


Moreover, \begin{eqnarray}&&\nabla_{\left[ \phi \frac{\partial \mathbf{r}(\hat{\mathbf{u}})}{\partial u_i},\frac{\partial \mathbf{r}(\hat{\mathbf{u}})}{\partial u_j}\right]}
\left(\phi^* \frac{\partial \mathbf{r}(\hat{\mathbf{u}})}{\partial u_k}\right)
 \nonumber \\ &=& \nabla_{\left(-\frac{\partial \mathbf{r}(\hat{\mathbf{u}})}
{\partial u_j}\cdot \phi \right) \frac{\partial \mathbf{r}(\hat{\mathbf{u}})}{\partial u_i}}\left(\phi^*\frac{\partial \mathbf{r}(\hat{\mathbf{u}})}
{\partial u_k}\right)  \nonumber \\ &=& -\nabla_{\left(\frac{\partial \phi}{\partial u_j}\frac{\partial \mathbf{r}(\hat{\mathbf{u}})}
{\partial u_i}\right) }\left(\phi^*\frac{\partial \mathbf{r}(\hat{\mathbf{u}})}{\partial u_k}\right) \nonumber \\ &=&
-\frac{\partial \phi}{\partial u_j} \nabla_{\frac{\partial \mathbf{r}(\hat{\mathbf{u}})}
{\partial u_i}}\left(\phi^*\frac{\partial \mathbf{r}(\hat{\mathbf{u}})}{\partial u_k}\right)
\nonumber \\ &=& -\frac{\partial \phi}{\partial u_j}\;\frac{\partial \phi^*}{\partial u_i}\;\frac{\partial \mathbf{r}(\hat{\mathbf{u}})}{\partial u_k}
- \frac{\partial \phi}{\partial u_j}\; \phi^*\nabla_{\frac{\partial \mathbf{r}(\hat{\mathbf{u}})}
{\partial u_i}}\frac{\partial \mathbf{r}(\hat{\mathbf{u}})}{\partial u_k}\nonumber \\ &=&
-\frac{\partial \phi}{\partial u_j}\;\frac{\partial \phi^*}{\partial u_i}\;\frac{\partial \mathbf{r}(\hat{\mathbf{u}})}{\partial u_k}
- \frac{\partial \phi}{\partial u_j}\; \phi^*\Gamma_{ik}^l\frac{\partial \mathbf{r}(\hat{\mathbf{u}})}{\partial u_l}
\nonumber \\ &=& -\frac{\partial \phi}{\partial u_j}\;\frac{\partial \phi^*}{\partial u_i}\;\delta_{kl}\;\frac{\partial \mathbf{r}(\hat{\mathbf{u}})}{\partial u_l}
- \frac{\partial \phi}{\partial u_j}\; \phi^*\Gamma_{ik}^l\frac{\partial \mathbf{r}(\hat{\mathbf{u}})}{\partial u_l}.
\end{eqnarray}

Thus,
\begin{eqnarray}R_{ijk}^l(\phi)&=&\phi\;\frac{\partial^2 \phi^*}{\partial u_i\partial u_j}\;\delta_{kl}
+\phi\; \frac{\partial \phi^*}{\partial u_j} \Gamma_{ik}^l
+ \phi\; \frac{\partial (\phi^* \Gamma_{jk}^l)}{\partial u_i}
+|\phi|^2\;\Gamma_{jk}^p \;\Gamma_{ip}^l\nonumber \\ &&
-\phi\;\frac{\partial^2 \phi^*}{\partial u_i\partial u_j}\;\delta_{kl}
-\phi \frac{\partial \phi^*}{\partial u_i} \Gamma_{jk}^l- \phi\; \frac{\partial (\phi^* \Gamma_{ik}^l)}{\partial u_j}
-|\phi|^2\;\Gamma_{ik}^p \;\Gamma_{jp}^l\nonumber \\ &&
+\frac{\partial \phi}{\partial u_j}\;\frac{\partial \phi^*}{\partial u_i} \delta_{kl}
+\frac{\partial \phi}{\partial u_j}\; \phi^*\Gamma_{ik}^l.
\end{eqnarray}

Simplifying this last result, we obtain
\begin{eqnarray}R_{ijk}^l(\phi)&=&\phi \;\frac{\partial \phi^*}{\partial u_j} \Gamma_{ik}^l
+ \phi\; \frac{\partial (\phi^* \;\Gamma_{jk}^l)}{\partial u_i}
-\phi\; \frac{\partial \phi^*}{\partial u_i} \Gamma_{jk}^l
 - \phi\; \frac{\partial (\phi^* \Gamma_{ik}^l)}{\partial u_j} \nonumber \\ &&
+|\phi|^2\left(\Gamma_{jk}^p \;\Gamma_{ip}^l-\Gamma_{ik}^p \;\Gamma_{jp}^l\right)+\frac{\partial \phi}{\partial u_j}\;\frac{\partial \phi^*}{\partial u_i} \delta_{kl}
+\frac{\partial \phi}{\partial u_j}\; \phi^*\Gamma_{ik}^l \nonumber \\ &=& |\phi|^2 \left(\frac{\partial \Gamma_{jk}^l}{\partial u_i}-\frac{\partial \Gamma_{ik}^l}{\partial u_j}+\Gamma_{jk}^p \;\Gamma_{ip}^l-\Gamma_{ik}^p \;\Gamma_{jp}^l\right) \nonumber \\ &&+\frac{\partial \phi}{\partial u_j}\;\frac{\partial \phi^*}{\partial u_i} \delta_{kl}
+\frac{\partial \phi}{\partial u_j}\; \phi^*\Gamma_{ik}^l \nonumber \\ &=&
|\phi|^2 \hat{R}_{ijk}^l\nonumber \\ &&+\frac{\partial \phi}{\partial u_j}\;\frac{\partial \phi^*}{\partial u_i} \delta_{kl}
+\frac{\partial \phi}{\partial u_j}\; \phi^*\Gamma_{ik}^l,
\end{eqnarray}
where $$\hat{R}_{ijk}^l=\frac{\partial \Gamma_{jk}^l}{\partial u_i}-\frac{\partial \Gamma_{ik}^l}{\partial u_j}+\Gamma_{jk}^p \;\Gamma_{ip}^l-\Gamma_{ik}^p \;\Gamma_{jp}^l$$ represents the  Riemann curvature tensor.

At this point, we recall to have defined this energy part by
$$
E_q= \frac{\gamma}{2}\int_0^T\int_\Omega R\;\sqrt{-g}\;|\det(\hat{\mathbf{u}}'(\mathbf{x},t))|\;d\mathbf{x}dt,$$
where $$R=g^{jk}R_{jk},$$ and as above indicated, $R_{jk}=Re[R_{jik}^i(\phi)].$

Hence the final expression for the energy (action) is given by

\begin{eqnarray}J(\phi,\mathbf{r}, \hat{\mathbf{u}},E)&=&-E_c+E_q \nonumber \\ &&-\int_0^T E(t)\left( \int_\Omega |\phi(\hat{\mathbf{u}}(\mathbf{x},t))|^2\;\sqrt{-g}\; |\det(\mathbf{u}'(\mathbf{x},t))|\;
d\mathbf{x}/c-1\right)\;cdt,\end{eqnarray}
where $E(t)$ is a Lagrange multiplier related to the total mass constraint.

More explicitly, the final action (the generalized Einstein-Hilbert one), would be given by

\begin{eqnarray}
 &&J(\phi,\mathbf{r}, \hat{\mathbf{u}},E)\nonumber \\ &=&\int_0^T \int_\Omega mc\;\sqrt{-g_{ij}\frac{\partial u_i}{\partial t} \frac{\partial u_j}{\partial t}}\;|\phi(\hat{\mathbf{u}}(\mathbf{x},t))|^2\; \sqrt{-g}\;|\det (\hat{\mathbf{u}}'(\mathbf{x},t))|\;d\mathbf{x}\;dt
 \nonumber \\ &&+\frac{\gamma}{2}\int_0^T\int_\Omega g^{jk}\;\frac{\partial \phi }{\partial u_j}\;\frac{\partial \phi^* }{\partial u_k}\;\sqrt{-g}\;|\det(\hat{\mathbf{u}}'(\mathbf{x},t))|\;d\mathbf{x}dt
\nonumber \\ &&+
\frac{\gamma}{4}\int_0^T\int_\Omega g^{jk}\;\left(\phi^* \frac{\partial \phi }{\partial u_l}+\phi \frac{\partial \phi^* }{\partial u_l}\right)\; \Gamma_{jk}^l\;\sqrt{-g}\;|\det(\hat{\mathbf{u}}'(\mathbf{x},t))|\;d\mathbf{x}dt
 \nonumber \\ &&  +\frac{\gamma}{2}\int_0^T\int_\Omega |\phi|^2\;g^{jk}\;\left(\frac{\partial \Gamma_{lk}^l}{\partial u_j}
 -\frac{\partial \Gamma_{jk}^l}{\partial u_l}+\Gamma_{lk}^p \;\Gamma_{jp}^l-\Gamma_{jk}^p \;\Gamma_{lp}^l\right) \;\sqrt{-g}\;|\det(\hat{\mathbf{u}}'(\mathbf{x},t))|\;d\mathbf{x}dt \nonumber \\ &&-\int_0^T E(t)\left( \int_\Omega |\phi(\hat{\mathbf{u}}(\mathbf{x},t))|^2\;\sqrt{-g}\; |\det(\hat{\mathbf{u}}'(\mathbf{x},t))|\;
d\mathbf{x}/c-1\right)\;cdt.
 \end{eqnarray}

Where  $\gamma$ is an appropriate positive constant to be specified.
\section{Obtaining the relativistic Klein-Gordon equation as an approximation of the previous action}

In particular for the special case in which $$\mathbf{r}(\hat{\mathbf{u}}(\mathbf{x},t))=\hat{\mathbf{u}}(\mathbf{x},t) \approx (ct,\mathbf{x}),$$ so that $$\frac{d \mathbf{r}(\hat{\mathbf{u}}(\mathbf{x},t))}{d t}\approx (c,0,0,0),$$
  we would  obtain $$\mathbf{g}_0\approx (1,0,0,0),\; \mathbf{g}_1\approx (0,1,0,0),\;\mathbf{g}_2\approx (0,0,1,0) \text{ and } \mathbf{g}_3\approx (0,0,0,1) \in \mathbb{R}^4,$$ and $\Gamma_{ij}^k\approx 0,\; \forall i,j,k \in \{0,1,2,3\}.$

Therefore, denoting $\phi(\hat{\mathbf{u}}(\mathbf{x},t))\approx \phi(ct,\mathbf{x})$ simply by a not relabeled $\phi(\mathbf{x},t),$ we may obtain 
\begin{eqnarray} E_q/c &\approx& \frac{\gamma  }{2}\int_0^T\int_\Omega
\left(-\frac{1}{c^2} \frac{\partial \phi(\mathbf{x},t) }{\partial t} \frac{\partial \phi^*(\mathbf{x},t) }{\partial t}
 \right.\nonumber \\ &&\left.+ \sum_{k=1}^3\frac{\partial \phi(\mathbf{x},t)}{\partial x_k} \frac{\partial \phi^*(\mathbf{x},t)}{\partial x_k} \right)\;d\mathbf{x}dt, \end{eqnarray}
 and
 $$E_c/c=m\;c^2\int_0^T\int_\Omega  |\phi|^2 \sqrt{1-v^2/c^2}\;\sqrt{-g}|\det(\hat{\mathbf{u}}'(\mathbf{x},t))|\;d\mathbf{x}\;dt/c \approx m c^2\int_0^T\int_\Omega |\phi(\mathbf{x},t)|^2\;d\mathbf{x}dt.$$

Hence, we would also obtain
\begin{eqnarray}J(\phi,\mathbf{r},\hat{\mathbf{u}},E)/c&\approx& \frac{\gamma}{2}\left(\int_0^T\int_\Omega
-\frac{1}{c^2} \frac{\partial \phi(\mathbf{x},t) }{\partial t} \frac{\partial \phi^*(\mathbf{x},t) }{\partial t}\;d\mathbf{x}dt
 \right.\nonumber \\ &&\left.+\sum_{k=1}^3\int_\Omega \int_0^T \frac{\partial \phi(\mathbf{x},t)}{\partial x_k} \frac{\partial \phi^*(\mathbf{x},t)}{\partial x_k}
 \;d\mathbf{x}dt\right) \nonumber \\ &&+ mc^2 \int_0^T\int_\Omega |\phi(\mathbf{x},t)|^2\;d\mathbf{x}dt  \nonumber \\ &&-\int_0^TE(t)\left( \int_\Omega |\phi(\mathbf{x},t)|^2d\mathbf{x}-1\right)\;dt.
\end{eqnarray}

The Euler Lagrange equations for such an energy are given by

\begin{eqnarray}\label{US34}
&&\frac{\gamma}{2}\left(\frac{1}{c^2} \frac{\partial^2 \phi(\mathbf{x},t)}{\partial t^2}- \sum_{k=1}^3 \frac{\partial^2 \phi(\mathbf{x},t)}{\partial x_k^2} \right)
\nonumber \\ &&+mc^2\phi(\mathbf{x},t)- E(t) \phi(\mathbf{x},t)=0, \text{ in } \Omega,
\end{eqnarray}
where we assume the space of admissible functions is given by $C^1(\Omega \times [0,T];\mathbb{C})$ with the following time and spatial boundary conditions,
$$\phi(\mathbf{x},0)=\phi_0(\mathbf{x}), \text{ in } \Omega,$$
$$\phi(\mathbf{x},T)=\phi_1(\mathbf{x}), \text{ in } \Omega,$$
$$\phi(\mathbf{x},t)=0, \text{ on } \partial\Omega \times[0,T].$$

Equation (\ref{US34}) is the relativistic Klein-Gordon one.

For $E(t)=E \in \mathbb{R}$ (not time dependent), at this point we suggest a solution (and implicitly related time boundary conditions) $\phi(\mathbf{x},t)=e^{-\frac{i E t}{\hbar}}\phi_2(\mathbf{x}),$ where $$\phi_2(\mathbf{x})=0, \text{ on } \partial \Omega.$$

Therefore, replacing this solution into equation (\ref{US34}), we would obtain
$$\left(\frac{\gamma}{2}\left(-\frac{E ^2}{c^2\hbar^2}\phi_2(\mathbf{x}) -\sum_{k=1}^3 \frac{\partial^2 \phi_2(\mathbf{x})}{\partial x_k^2}\right)+mc^2\phi_2(\mathbf{x})- E \phi_2(\mathbf{x})\right)e^{-\frac{i  E  t}{\hbar}}=0,$$
$\text{ in } \Omega.$

Denoting $$E_1=-\frac{\gamma  E^2}{2c^2 \hbar^2}+mc^2-E,$$ the final eigenvalue problem would stand for
$$-\frac{\gamma}{2} \sum_{k=1}^3 \frac{\partial^2 \phi_2(\mathbf{x})}{\partial x_k^2} +E_1\phi_2(\mathbf{x})=0, \text{ in } \Omega$$
where $E_1$ is such that
$$\int_\Omega |\phi_2(\mathbf{x})|^2\;d\mathbf{x}=1.$$

Moreover, from (\ref{US34}), such a solution $\phi(\mathbf{x},t)=e^{-\frac{i  E t}{\hbar}}\phi_2(\mathbf{x})$ is also such that
\begin{eqnarray}\label{US35}
&&\frac{\gamma}{2}\left(\frac{1}{c^2} \frac{\partial^2 \phi(\mathbf{x},t)}{\partial t^2}- \sum_{k=1}^3 \frac{\partial^2 \phi(\mathbf{x},t)}{\partial x_k^2} \right)
\nonumber \\ &&+mc^2\phi(\mathbf{x},t)=i  \hbar\frac{\partial \phi(\mathbf{x},t)}{\partial t}, \text{ in } \Omega.
\end{eqnarray}
At this point, we recall that in quantum mechanics, $$\gamma=\hbar^2/m.$$

Finally, we remark this last equation (\ref{US35}) is a kind of relativistic Schr\"{o}dinger-Klein-Gordon equation.

\section{ A note on the Einstein field equations in the vacuum}

In this section we obtain the Einstein field equations for a field of position in the vacuum.

Let $\Omega \subset \mathbb{R}^3$ be an open, bounded and connected set with a regular boundary denoted by $\partial \Omega$. Let $[0,T]$ be a time interval and
consider the Hilbert-Einstein action given by
$J:U \rightarrow \mathbb{R}$, where for an appropriate constant $\gamma>0$,
$$J(\mathbf{r})=\frac{\gamma}{2} \int_0^T \int_\Omega \hat{R} \sqrt{-g}\; d \mathbf{x} dt,$$
where again
$$\mathbf{u}=(u_0,u_1,u_2,u_3)=(t,x_1,x_2,x_3)=(x_0,x_1,x_2,x_3).$$

Also, $$g_{jk}=\frac{\partial \mathbf{r}(\mathbf{u})}{\partial u_j} \cdot  \frac{\partial \mathbf{r}(\mathbf{u})}{\partial u_k},$$ $$g=det\{g_{jk}\},$$
$$\hat{R}=g^{jk}R_{jk},$$
$$R_{ik}=R_{ijk}^j,$$
and
$$R_{ijk}^l=\frac{\partial \Gamma_{jk}^l}{\partial u_i}-\frac{\partial \Gamma_{ik}^l}{\partial u_j}+\Gamma_{jk}^p \;\Gamma_{ip}^l-\Gamma_{ik}^p \;\Gamma_{jp}^l$$ represents the Riemann curvature tensor.

Finally, as above indicated, $$\mathbf{r}:\Omega \times [0,T] \rightarrow \mathbb{R}^4$$ stands for
$$\mathbf{r}(\mathbf{u})=(ct,X_1(\mathbf{u}),X_2(\mathbf{u}),X_3(\mathbf{u}))$$ and
\begin{eqnarray}U&=&\left\{ \mathbf{r} \in W^{2,2}(\Omega;\mathbb{R}^4) \;:\;  \mathbf{r}(0,u_1,u_2,u_3)=\mathbf{r}_1(u_1,u_2,u_3),
\right.\nonumber \\ && \left.\mathbf{r}(c T,u_1,u_2,u_3)=\mathbf{r}_2(u_1,u_2,u_3) \text{ in } \Omega,\;
 \mathbf{r}|_{\partial \Omega}=\mathbf{r}_0 \text{ on } [0,T] \right\}.
\end{eqnarray}

Hence, already including the Lagrange multipliers, considering $\mathbf{r}$ and $\{g_{jk}\}$ as independent variables, such a functional again denoted by $J(\mathbf{r}, \{g_{jk}\},\lambda)$ is expressed as
$$J(\mathbf{r},\{g_{jk}\},\lambda)=\frac{\gamma}{2} \int_0^T \int_\Omega \hat{R} \sqrt{-g}\;d\mathbf{x}dt +\int_0^T\int_\Omega  \lambda_{jk}\left(\frac{\partial \mathbf{r}}{\partial u_j}
\cdot \frac{\partial \mathbf{r}}{\partial u_k}-g_{jk} \right)\;d \mathbf{x}dt.$$

The variation of such a functional in $g$ give us
$$\gamma \left(R_{jk}-\frac{1}{2}g_{jk}\hat{R}\right)\sqrt{-g}-\lambda_{jk}=0, \text{ in } \Omega.$$

The variation in $\mathbf{r}$, provide us
\begin{equation}\frac{\partial^2X_l(\mathbf{u})}{\partial u_j \partial u_k} \lambda_{jk}+\frac{\partial X_l(\mathbf{u})}{\partial u_j} \frac{\partial \lambda_{jk}}{\partial u_k}=0,
\text{ in } \Omega,\end{equation}
so that
\begin{equation}\label{s1}\frac{\partial^2X_l(\mathbf{u})}{\partial u_j \partial u_k} \left( [R_{jk}-\frac{1}{2}g_{jk}\hat{R}]\sqrt{-g}\right)+\frac{\partial X_l(\mathbf{u})}{\partial u_j} \frac{\partial [(R_{jk}-\frac{1}{2}g_{jk}\hat{R})\sqrt{-g}]}{\partial u_k}=0,
\text{ in } \Omega,\end{equation}

$\forall l \in \{1,2,3\}.$

Observe that the condition $R_{jk}=0 \text{ in } \Omega \times [0,T],\; \forall j,k \in \{0,1,2,3\}$, it is sufficient to solve the system indicated in (\ref{s1}) but it is not necessary.

The system indicated in (\ref{s1}) is the Einstein field one. It is my understanding the actual variable for this system is $\mathbf{r}$ not $\{g_{jk}\}$.

However in some situations, it is possible to solve (\ref{s1}) through a specific metric  $\{(g_0)_{jk}\}$, but one question remains, how to obtain a corresponding $\mathbf{r}.$

With such an issue in mind, given a specific metric $\{(g_0)_{jk}\}$, we suggest the following control problem,
$$\text{ Find } \mathbf{r} \in U \text{ which } \text{ minimizes } J_1(\mathbf{r})= \sum_{j,k=0}^3\left\| \frac{\partial \mathbf{r}(\mathbf{u})}{\partial u_j} \cdot \frac{\partial \mathbf{r}(\mathbf{u})}{\partial u_k} -(g_0)_{jk} \right\|_2^2,$$
subject to
\begin{equation}\frac{\partial^2X_l(\mathbf{u})}{\partial u_j \partial u_k} \left( [R_{jk}-\frac{1}{2}g_{jk}\hat{R}]\sqrt{-g}\right)+\frac{\partial X_l(\mathbf{u})}{\partial u_j} \frac{\partial [(R_{jk}-\frac{1}{2}g_{jk}\hat{R}) \sqrt{-g}]}{\partial u_k}=0,
\text{ in } \Omega,\end{equation}
$\forall l \in \{1,2,3\}.$

\section{Conclusion}This work proposes an action (energy) suitable for the relativistic quantum mechanics context. The Riemann tensor
represents an important part of the action in question, but now including the density distribution of mass in its expression.
In one of the last sections, we obtain the relativistic Klein-Gordon equation as an approximation of the main action, under specific properly described conditions.

We believe the results obtained may be applied to more general models, such as those involving atoms and molecules subject to the presence of electromagnetic fields.

Anyway, we postpone the development of such studies for a future research.

\end{document}